\documentclass[12pt,a4paper]{amsart}

\usepackage[utf8]{inputenc}
\usepackage{amsmath}
\usepackage{amsfonts}
\usepackage{amssymb}
\usepackage{amsthm}
\usepackage[margin=2cm]{geometry}
\usepackage[utf8]{inputenc}
\usepackage{amsmath}
\usepackage{amsfonts}
\usepackage{amssymb}
\usepackage{amsthm}
\usepackage{geometry}

\usepackage{mathtools}
\usepackage[dvipsnames]{xcolor}
\usepackage{url}
\usepackage{enumitem}
\usepackage{float}
\usepackage{graphicx}
\usepackage{multirow}
\usepackage[hypertexnames=false]{hyperref}
\newcommand{\RR}{\mathbb{R}}

\newcommand{\NN}{\mathbb{N}}

\newcommand{\eps}{\varepsilon}

\newcommand{\dx}{ \mathrm{dx}}
\newcommand{\dy}{ \mathrm{dy}}

\newtheorem{theorem}{Theorem}[section]

\newtheorem{proposition}[theorem]{Proposition}
\newtheorem{lemma}[theorem]{Lemma}

\theoremstyle{definition}

\newtheorem{remark}[theorem]{Remark}
\numberwithin{equation}{section}

\begin{document}
	 
	\title[Hardy inequality and heat equation on $\RR^{N-k}\times (0,\infty)^k$]{The Hardy inequality and large time behaviour of the heat equation on $\RR^{N-k}\times (0,\infty)^k$}
	\author{Cristian Cazacu}
	\address{Cristian Cazacu: $^1$Faculty of Mathematics and Computer Science \\
		University of Bucharest\\
		010014 Bucharest, Romania\\
		\&
		$^2$Gheorghe Mihoc-Caius Iacob Institute of Mathematical
		Statistics and Applied Mathematics of the Romanian Academy\\
		050711 Bucharest, Romania
	}
	\email{cristian.cazacu@fmi.unibuc.ro}
	\author{Liviu I. Ignat}
	\address{Liviu I. Ignat: $^1$"Simion Stoilow" Institute of Mathematics of the Romanian Academy\\
		Calea Grivitei Street, no. 21\\
		010702 Bucharest, Romania\\
		\&
		$^2$ICUB-Research Institute of the University of Bucharest\\
		University of Bucharest\\
		Șoseaua Panduri, no. 90, Sector 5\\ 050663, Bucharest, Romania  
	}
	\email{liviu.ignat@gmail.com}
	\author{Drago\c s Manea}
	\address{Drago\c s Manea: "Simion Stoilow" Institute of Mathematics of the Romanian Academy\\
		Calea Grivitei Street, no. 21\\
		010702 Bucharest, Romania
	}
	\email{dmanea28@gmail.com}
	\thanks{C. C. and L. I. were partially supported by the Romanian Ministry of Research, Innovation and Digitization,
		CNCS - UEFISCDI, project number PN-III-P1-1.1-TE-2021-1539, within PNCDI III. D. M. was partially supported by CNCS-UEFISCDI Romania, Grant no. 0794/2020
“Spectral Methods in Hyperbolic Geometry” PN-III-P4-ID-PCE-2020-0794 }
	\date{\today}
	
	\begin{abstract}
		\noindent In this paper we study the large time asymptotic behaviour of the heat equation with Hardy inverse-square potential on corner spaces $\RR^{N-k}\times (0,\infty)^k$, $k\geq 0$. We first show a new improved  Hardy-Poincar\'{e} inequality for the quantum harmonic oscillator with Hardy potential. In view of that,  we construct the appropriate functional setting in order to pose the Cauchy problem. Then 
		we obtain optimal polynomial large  time decay rates and subsequently the first term in the asymptotic expansion of the solutions in $L^2(\RR^{N-k}\times (0,\infty)^k)$. Particularly, we  extend and improve the results obtained by V\'{a}zquez and Zuazua (J. Funct. Anal. 2000), which correspond to the case $k=0$, to any $k\geq 0$. We emphasize that the higher the value of $k$ the better time decay rates are. We employ a different and simplified approach than V\'{a}zquez and Zuazua, managing to remove the usage of spherical harmonics decomposition in our analysis. 
	\end{abstract}
	\maketitle
	\section{Introduction}
	In this paper we analyze sharp time asymptotic properties of the Cauchy problem with Dirichlet boundary conditions:
	\begin{equation}
		\label{heatHardyHalf}
		\left\{\begin{array}{ll}
		\partial_t u(t,x)=\Delta u(t,x)+\frac{\lambda}{|x|^2}u(t,x), & t> 0, x\in \RR^{N, k}_+, \\
			u(t,x)=0, & t> 0,\  x\in \partial \RR^{N, k}_+, \\
			u(0,x)=u_0(x), & x\in \RR^{N, k}_+,
		\end{array}\right.
	\end{equation}
	where $N\geq 2$ is the space dimension,  $k\in \{0, \ldots, N\}$ is a fixed integer and $\RR_{+}^{N, k}$ denotes the space $\RR^{N-k}\times (0,\infty)^k\subset \RR^N$ with  its boundary denoted by $\partial\RR_{+}^{N, k}$.  The initial data $u_0$ is assumed to belong to $L^2(\RR_{+}^{N, k})$, $|x|$ is the Euclidean norm of $x$ in $\RR^N$ and $\lambda/|x|^2$ is an inverse-square Hardy potential with a  real parameter $\lambda$ which will be made precise later. To be more specific, for $N\in \NN$ and $k\in \{0,1,\ldots,N\}$, we denote
	\begin{multline*}
		\RR_{+}^{N, k}:=\Big\{x=(x^\prime, x_{N-k+1}, \ldots, x_N) \in \RR^N \ |\ x^\prime=(x_1, \ldots, x_{N-k})\in \RR^{N-k} \\
		\mathrm{ and } \        x_{N-k+j}>0, \mathrm{for\ all \ } j\in \NN \ \mathrm{ with \ } 1\leq j\leq k \Big\}
	\end{multline*}
	and 
	\begin{multline*}
		\partial\RR_{+}^{N, k}=\Big\{x=(x^\prime, x_{N-k+1}, \ldots, x_N) \ |\ x^\prime=(x_1, \ldots, x_{N-k})\in \RR^{N-k} \\
		\mathrm{ and } \        x_{N-k+j}=0, \mathrm{for\ all \ } j\in \NN \ \mathrm{ with \ } 1\leq j\leq k  \Big\}.
	\end{multline*} 
	It can be easily seen that $\RR_{+}^{N, 0}$ coincides with the space $\RR^N$ and $\RR_{+}^{N, 1}$ is known as the upper half space, usually denoted by $\RR_{+}^N$ in the literature.
	
	The real parameter $\lambda$  in the  singular potential $\lambda/|x|^2$ is assumed to belong to the range $\lambda\in (-\infty, \lambda_{N, k}]$ where 
	\begin{equation}\label{constant}
		\lambda_{N, k}:= \left (\frac{N-2}{2}+k\right)^2
	\end{equation}
	is   the optimal constant in the Hardy inequality on $\RR_{+}^{N, k}$. This inequality (see e.g. \cite{Su-Yang})  reads:
	
	\begin{equation}\label{Hardy_k}
		\int_{\RR_{+}^{N, k}} |\nabla u|^2 \dx   \geq \lambda_{N, k}\int_{\RR_{+}^{N, k}}\frac{u^2}{|x|^2} \dx, \quad \forall u \in C_c^\infty(\RR_{+}^{N, k}).
	\end{equation}

	Later in  paper, we will construct the appropriate functional setting for the problem to be well-posed  and study the asymptotic behaviour of the solution in the $L^2$ norm. We distinguish two different situations:
	\begin{enumerate}[label=\alph*)]
	\item the subcritical case $\lambda< \lambda_{N, k}$;
	\item the critical case $\lambda= \lambda_{N, k}$.
	\end{enumerate}

	In order to outline the main results and  tools used in the proofs, let us consider the Hamiltonian operator $A_\lambda$, initially defined on $C_c^\infty(\RR_{+}^{N, k})$, which has the form:
	$$A_\lambda u:=-\Delta u -\frac{\lambda}{|x|^2}u.
	$$
	Then, the problem \eqref{heatHardyHalf} can be written as:
	\begin{equation}
		\label{heatHamitonian}
		\left\{\begin{array}{ll}
			\partial_t u(t,x)+A_\lambda u(t, x)=0, & t> 0, x\in \RR^{N, k}_+, \\
			u(t,x)=0, & t> 0,\  x\in \partial \RR^{N, k}_+, \\
			u(0,x)=u_0(x), & x\in \RR^{N, k}_+.
		\end{array}\right.
	\end{equation}
Due to the Hardy inequality \eqref{Hardy_k},	a theorem by Reed and Simon \cite[cf. Theorem VIII.15]{ReedSimon} allows us to extend $A_\lambda$ for any $\lambda\leq \lambda_{N,k}$ to a self-adjoint unbounded and nonnegative operator on $L^2(\RR_{+}^{N,k})$.  
	So, following The Hille-Yosida theory (see, e.g. \cite{cazenave}), we can consider the contractions semigroup
	$$e^{-tA_\lambda}: L^2(\RR_{+}^{N, k})\rightarrow L^2(\RR_{+}^{N, k}), \quad  u_0 \mapsto u(t):=e^{-t A_\lambda}u_0, \quad t\geq 0,$$
	satisfying $u\in C([0, \infty); L^2(\RR_{+}^{N, k}))$. 
The contraction property implies
     $\|u(t)\|_{L^2(\RR_+^{N, k})}\leq \|u_0\|_{L^2(\RR_+^{N, k})}$ and moreover, for any positive time $t$ we have $\|e^{-t A_\lambda}\|_{L^2(\RR_+^{N, k})\rightarrow L^2(\RR_+^{N, k})}= 1$.   
     To check this it is enough to consider a sequence of initial data $(\varphi_n)_{n\geq 1}$ in $D(A_\lambda)$ such that $\|\varphi_n\|_{L^2(\RR_+^{N, k})}=1$ and 
     $\|A_\lambda \varphi_n\|_{L^2(\RR_+^{N, k})}\rightarrow 0$, i.e. $\varphi_n(x)=(1/n)^{N/2} \varphi(x/n)$ with $\varphi$ a smooth normalized compactly supported function.

	One way we can force the semigroup to decay is to consider the initial data in a subspace of $L^2(\RR_{+}^{N,k})$.  It is by now classical to take $u_0$ in the weighted space 
	$$L^2(\RR_{+}^{N, k}; K):=\left\{u \in L^2(\RR_{+}^{N, k}) \  \Big|\ \int_{\RR_{+}^{N, k}} |u(x)|^2 K(x) \dx <\infty \right\}, \quad K(x):=e^{\frac{|x|^2}{4}}.$$
	
	In the following, where there is no risk of confusion, we write $L^2$ and $L^2(K)$ instead of $L^2(\RR_{+}^{N, k})$ and $L^2(\RR_{+}^{N, k}; K)$, respectively. Then we initiate our analysis by using the so-called "self-similiary variables" method, which will allow us to cast problem \eqref{heatHardyHalf} into another one, induced by an operator with compact inverse, which will help to gain polynomial decay rates in time for the original problem \eqref{heatHamitonian}. 
    
    The self-similarity variables transformation we use takes the following form:
    \begin{equation}
        \label{selfSimilarityFull}
        v(s,y)=K^{\frac 1 2}(y) e^{\frac{Ns}{4}} u\left(e^s-1,e^{\frac{s}{2}}y\right).
    \end{equation}
    With this notation, $v$ formally satisfies the following Cauchy problem:
    
    \begin{equation}
		\label{hardyPotentialSymFull}
		\left\{\begin{array}{ll}
			\partial_s v(s,y)+ A_\lambda v(s, y) +\frac{|y|^2}{16}v(s,y)=0, & s> 0, \ y\in \RR_{+}^{N, k}, \\
			v(s,y)=0, & s>0, y\in \partial \RR^{N,k}_+,\\
			v(0,y)=v_0(y)=K^{\frac 1 2 }(y)u_0(y), & y\in \RR_{+}^{N, k}.
		\end{array}\right.
	\end{equation}
    The evolution problem \eqref{hardyPotentialSymFull} is a contraction semigroup in $L^2$ generated by the operator 
     \begin{equation}
	\label{LLambdaFormal}
	L_\lambda v= A_\lambda v + \frac{|y|^2}{16}v=-\Delta v-\frac{\lambda}{|y|^2}v+ \frac{|y|^2}{16}v,
	\end{equation}
 which is defined through the associated $L^2$ quadratic form $l_\lambda[\cdot]$,
    \begin{equation}
        \label{intro.defllambdaQuadratic}
        l_\lambda[v] :=\int_{\RR_+^{N, k}} \left(|\nabla v|^2  - \frac{\lambda}{|y|^2}|v|^2 + \frac{|y|^2}{16}|v|^2\right)\dy, \quad v\in C_c^\infty(\RR^{N,k}_+).
    \end{equation}
    The domain  of $l_\lambda[\cdot]$  is given by:
	\begin{equation}
	    \label{intro.defH}
		H_\lambda=D(l_\lambda):=\overline{C_c^\infty(\RR_{+}^{N, k})}^{\|\cdot\|_{l_\lambda}},
	\end{equation}
	where the norm $\|\cdot\|_{l_\lambda}$ on which the closure is taken is given by $\|v\|_{l_\lambda}:=\sqrt{l_\lambda[v]}$. This is indeed a norm since it controls the $L^2$ norm (see Theorem \ref{main_th1} below). Moreover, $H_{\lambda}$ is continuously embedded in $L^2$, hence \cite[Theorem VIII.15]{ReedSimon} allows us to define an unbounded operator $L_\lambda$ on $L^2(\RR^{N,k}_+)$, with domain $D(L_\lambda)$.	A detailed analysis of the spaces $H_\lambda$ and $D(L_\lambda)$ is done in Section  \ref{functionalSpectral}.

Since the operator $L_\lambda$ is self-adjoint and non-negative, the well-posedness of the problem \eqref{equationVHat} is ensured by the Hille-Yosida theory (see, e.g. \cite[Theorem 3.2.1]{cazenave})  which leads to the following: for any $v_0\in L^2(\RR_+^{N,k})$ problem \eqref{hardyPotentialSymFull} has a unique solution:
	\begin{equation} \label{well-posedness}
	v\in C([0,\infty),L^2(\RR_+^{N,k}))\cap C((0,\infty),D(L_\lambda))\cap C^1((0,\infty),L^2(\RR_+^{N,k})).
\end{equation}

	  Regarding the problem \eqref{hardyPotentialSymFull}, we should remark two things. First, the assumption that $u_0$ is in $L^2(K)$, implies that the initial data $v_0$ is in $L^2$. Next, the extra term $\frac{|y|^2}{16} v$, which blows up for large $|y|$, forces the solutions to concentrate in compact regions, so we obtain compactness when embedding the form domain $H_\lambda$ into $L^2$. 
    This shows that the spectrum of the operator $L_\lambda$ is discrete and accumulates at infinity. Applying analytical tools  which are specific to parabolic equations, we deduce that the first eigenvalue of $L_\lambda$ is positive and provides sharp exponential decay for the solutions $v(s, \cdot)$ of \eqref{hardyPotentialSymFull}. In consequence, returning to the original problem, we are able to obtain sharp polynomial large time decay for the solution $u(t, \cdot)$ of \eqref{heatHamitonian}. This will be rigorously detailed later.    
	
		In order to state our main results we need to introduce the number  which quantifies the criticality of the parameter $\lambda$:
\begin{equation}\label{key_number}
m_\lambda:=\sqrt{\lambda_{N, k}-\lambda}
\end{equation}
	and the functions
	\begin{equation}\label{minimizers}
		\alpha_{k,\lambda}(x)=e^{-\frac{|x|^2}{8}}\, |x|^{m_\lambda-\frac{N-2}{2}}\, \frac{x_{N-k+1}x_{N-k+2}\cdots x_{N}}{|x|^k},
		\end{equation}
	which, as we will see later, play a fundamental role in our analysis.  
	
	The first main result of our paper relies on the  following improved Hardy-Poincar\'{e} inequality in which the sharp constant turns out to be the first eigenvalue of the operator $L_\lambda$.  
	\begin{theorem}\label{main_th1}
		For any $u \in C_c^\infty(\RR_{+}^{N, k})$ it holds 
		\begin{equation}\label{improved_Hardy}
			\int_{\RR_{+}^{N, k}}\left(|\nabla u|^2 -\lambda\frac{|u|^2}{|x|^2} + \frac{|x|^2}{16}|u|^2\right) \dx \geq \frac{1+m_\lambda}{2}\int_{\RR^N} |u|^2 \dx.
		\end{equation}
		Moreover, the constant $\frac{1+m_\lambda}{2}$ in \eqref{improved_Hardy} is sharp and it is achieved by the functions $\alpha_{k, \lambda}$ in \eqref{minimizers}. 
 is (see Section \ref{section:improvedHardy} for details).
	\end{theorem}
	
	\begin{remark}
		In fact, the minimizers $\alpha_{k, \lambda}$ in \eqref{improved_Hardy} do not belong to $C_c^\infty(\RR_{+}^{N, k})$ but will see that that inequality \eqref{improved_Hardy} extends to functions $u$ from the larger space $H_\lambda$ and the equality is achieved in $H_\lambda$ for $u=\alpha_{k, \lambda}$   (see Section \ref{functionalSpectral} for details).
	\end{remark}
      One of the main purposes of this paper is to find the following number:
	\begin{equation*}\label{decay_rate}
		\gamma_\lambda:=\sup \left\{ \gamma \geq 0 \ \Big |\  \exists C_\gamma >0 \ \mathrm{  s. t. }\   \|u(t)\|_{L^2}\leq C_\gamma (1+t)^{-\gamma}\|u_0\|_{L^2(K)}, \forall t\geq 0, u_0\in L^2(K)\right\}
	\end{equation*}
	which will provide us the optimal polynomial decay rate of the solutions of \eqref{heatHardyHalf}.
		In addition, we are interested in the asymptotic profile of the solution as time goes to infinity, namely in the behaviour of the function 
	$t\to t^{\gamma_\lambda}u(t).$
	The main result of the paper is the following:
	\begin{theorem}\label{main_th}
		For any initial data $u_0\in L^2(K)$, the solution $u$ of problem \eqref{heatHardyHalf} satisfies:
		\begin{equation}
			\label{solutionDecayRNk}
			\left\|u(t)\right\|_{L^2(\RR_{+}^{N,k})} \leq (1+t)^{-\frac{1+m_\lambda}{2}}\|u_0\|_{L^2(\RR_{+}^{N,k}; K)},\hspace{0.3cm}\forall t\geq 0.
		\end{equation}
		Moreover,	this polynomial decay rate is sharp:
		\begin{equation}
			\label{asymptoticProfileRNK}
			\lim_{t\to \infty} t^{\frac{1+m_\lambda}{2}} \left\|u(t)-\beta\, t^{-(1+m_\lambda)}e^{-\frac{|x|^2}{4t}} e^{\frac{|x|^2}{8}} \alpha_{k, \lambda}\right\|_{L^2(\RR_{+}^{N, k})}=0,
		\end{equation}
		where
		$$\beta=\left(\left\|\alpha_{k,\lambda}\right\|_{L^2(\RR_{+}^{N, k})}\right)^{-1}\int_{\RR_{+}^{N, k}} u_0(x)e^{\frac{|x|^2}{8}} \alpha_{k, \lambda}(x) \dx.$$
	\end{theorem}
	\begin{remark}
	    Formula \eqref{asymptoticProfileRNK} expresses the fact that $u$ admits the asymptotic expansion in $L^2$:
	    $$u(t, x)= U(t, x)+ o (t^{-\frac{1+m_\lambda}{2}}), \textrm{ as } t\rightarrow \infty.$$
	    where the asymptotic profile $U(t, x):=\beta\, t^{-(1+m_\lambda)}e^{-\frac{|x|^2}{4t}} e^{\frac{|x|^2}{8}} \alpha_{k, \lambda}(x)$ satisfies that $t^{\frac{1+m_\lambda}{2}}\|U(t, \cdot)\|_{L^2(\RR^{N,k}_+)}$ is constant in time.  
	\end{remark}
	
	\begin{remark}
		In particular, if $\lambda$ is critical, i.e. $\lambda=\lambda_{N, k}$, then $m_\lambda=0$, so the optimal decay rate in \eqref{solutionDecayRNk} is:
		$(t+1)^{-\frac{1}{2}}$ 
		and the normalizing constant is
		$$\beta=\sqrt{\frac{2^{k-1}}{\omega_{N,k}}}\int_{\RR^N_k} u_0(x)e^{\frac{|x|^2}{8}} \alpha_{k, \lambda}(x) \dx,$$
		where $\omega_{N,k}=\int_{S^{N-1}}\sigma_{N-k+1}^2\cdots \sigma_N^2 d\sigma$.
		More details about this calculations can be seen in Section \ref{section:asymptotic}.
	\end{remark}
	
	\begin{remark}
	   In the case of the half-space $\RR^N_+$, i.e. $k=1$, Theorem \ref{asymptoticProfileRNK} reads:
		\begin{equation}
			\label{solutionDecayHalf}
			 \left\|u(t)\right\|_{L^2(\RR^N_+)} \leq (1+t)^{-\frac{1+m_\lambda}{2}} \|u_0\|_{L^2(\RR^N_+,K)},\hspace{0.3cm}\forall t\geq 0
		\end{equation}
		and
		\begin{equation}
			\label{asymptoticProfileHalf}
			\lim_{t\to 0} t^{\frac{1+m_\lambda}{2}} \left\|u(t,x)-\beta\, t^{-(1+m_\lambda)}e^{-\frac{|x|^2}{4t}}\, |x|^{m_\lambda-\frac{N}{2}}\, x_N \right\|_{L^2\left(\RR^N_+\right)(x)}=0,
		\end{equation}
		where $m_\lambda=\sqrt{\frac{N^2}{4}-\lambda}$ and
		$$\beta=\frac{\sqrt{N}}{2^{m_\lambda}\sqrt{\Gamma(m_\lambda+1)}}\int_{\RR^N_+} u_0(x)\,|x|^{m_\lambda-\frac{N}{2}}\, x_N \dx,$$
			where $\Gamma$ denotes the Gamma function.  
		In particular, for critical $\lambda=\lambda_{N,1}=\frac{N^2}{4}$ we have $m_\lambda=0$ and 
		$$\beta=\sqrt{N}\int_{\RR^N_+} u_0(x)\,|x|^{-\frac{N}{2}}\, x_N \dx.$$
		\end{remark}
	\begin{remark}
		Inequality \eqref{solutionDecayHalf} implies that, if we work on the half-space $\RR^N_+$ with $\lambda=\frac{(N-2)^2}{4}$, which is the critical Hardy constant on $\RR^N$, we obtain the following decay estimate:
		$$t^{\frac{1+\sqrt{N-1}}{2}} \left\|u(t)\right\|_{L^2(\RR^N_+)} \leq \|u_0\|_{L^2(\RR^N_+,K)},\hspace{0.3cm}\forall t\geq 0,$$
		so the decay is improved from $\gamma_\lambda=\frac{1}{2}$ (as proved in \cite{VazquezZuazua}) to $\gamma_\lambda=\frac{1+\sqrt{N-1}}{2}$ when we switch from $\RR^N$ to $\RR^N_+$, while keeping the same parameter $\lambda$.
		More generally, when $\lambda$ is kept constant, the decay improves as we work on spaces $\RR_+^{N,k}$ with increasing $k$.
	\end{remark}

	The well-posedness of the heat equation with singular inverse-square potential was first addressed by Baras and Goldstein \cite{BarasGoldstein}, who showed that the Cauchy problem in $\RR^N$ is well-posed for $\lambda$ less or equal to the critical Hardy exponent $\lambda_{N,0}$, but ill-posed for greater values of $\lambda$. Later on, many paper studied well-posedness and asymptotic results for several generalisations of \eqref{heatHardyHalf}. For instance, Ferreira and Mesquita \cite{FerreiraMesquita} considered the heat equation with multipolar inverse-square potentials in $\RR^N$, proved sufficient conditions for well-posedness and analysed the asymptotic behaviour of the solutions via Fourier transform techniques. Abdellaoui, Peral and Primo \cite{AbdellaouiPeralPrimo} managed to prove well-posedness for any $\lambda>0$, provided that  an absorption term of the form $|\nabla u|^p$, for suitable $p$, is added to the equation \eqref{heatHardyHalf} Another direction of generalisation was replacing the heat operator in \eqref{heatHardyHalf} with the p-Laplacian: Garc\' ia Azorero and Peral Alonso \cite{GarciaAzoreroPeralAlonso} proved a well-posedness result in this case and Qian and Shen \cite{QianShen} proved the existence of a global attractor for the solutions. The heat equation with a perturbed version of the inverse-square potential in $\RR^N$ was studied by Ishige and Mukai \cite{IshigeMukai}, who obtained asymptotic profiles for initial data in $L^2(K)$, using spectral methods.
 
	The methods we use in our paper are based on the so-called self-similarity transform   (abv. SST),  an improved Hardy-type inequality and spectral analysis in $L^2(\RR_{+}^{N, k})$.  The SST method has been widely used in significant works in the literature since the '80s. For instance, Gidas and Spruck \cite{GidasSpruck} consider the self-similarity scaling to study the behaviour of elliptic equations near isolated singularities of the coefficient functions. Later on, Giga and Kohn \cite{GigaKohn} used the same technique to analyse the blow-up of solutions of the non-linear heat equation on a cylinder. Some years later, Escobedo and Kavian \cite{EscobedoKavian} considered SST and the weighted spaces $L^2(K)$ and $H^1(K)$ to study existence and uniqueness results for the semilinear heat equation. They proved the compactness of the embedding $ H^1(K) \subset L^2(K)$ and used spectral properties of the linear operator involved in the expression of the equation in similarity variables. These techniques were then applied by Escobedo and Zuazua \cite{EscobedoZuazua1991} in the analysis of the large-time asymptotic behaviour for the non-linear convection diffusion equation.
	
	The Hardy inequality is a cornerstone in the study of all sorts of differential equations, with a history of more than a century. A completely non-exhaustive survey of the extensions of this inequality would include the notorious results of Brezis-Marcus \cite{BM1997} and Brezis-Vazquez \cite{BV1997}, together with the more recent contributions of Ghoussoub-Moradifam \cite{GM2011} and 
 Lam, Lu and Zhang \cite{LLZ2020}.
 
  The contribution of Hardy-type inequalities in the study of asymptotic behaviour was the subject of various influential papers, starting with the  pioneering work of V\'azquez and Zuazua \cite{VazquezZuazua}, which deals with the heat equation with Hardy potential on $\RR^N$. They analysed the large-time decay and asymptotic profile of this equation, using, besides Hardy-type inequalities, all the machinery of SST and spectral results in the weighted space $L^2(K)$. Afterwards, those techniques were employed in order to obtain decay estimates for the heat equations in twisted tubes (Krejčiřík, Zuazua \cite{DavidZuazuaTwistedDomains}) and curved wedges (Krejčiřík \cite{DavidCurvedWedges}), but also for the heat semigroup with local magnetic field in the plane (Krejčiřík \cite{Krejcirik}).  Later on, the paper  \cite{CazacuDavid} by Cazacu and Krejčiřík is concerned with the  heat semigroup generated by the magnetic Schr\" odinger operator  perturbed by the critical Hardy potential.

It is also worth mentioning the work \cite{Gkikas}, where Gkikas obtained small-time kernel estimates and a Harnack inequality for the heat equation with inverse-squared-distance potential in exterior domains, via Hardy-Poincare inequalities.

	The main novel aspects brought in by our paper, compared to the references above, can be resumed as follows: 
 
 \begin{enumerate}[label=\alph*)]
 \item We generalize the asymptotic results for the heat equation with Hardy potential, obtained by V\'{a}zquez and Zuazua \cite{VazquezZuazua} in the case of the whole space $\RR^N$, to the corner space $\RR_{+}^{N, k}$, providing an improved decay;
 \item  Our analysis is more straightforward than its counterpart in \cite{VazquezZuazua}, as we totally get rid of the spherical harmonics machinery, especially in the proof of Theorem \ref{main_th1}, where we only have to properly design the positive eigenfunctions $\alpha_{k,\lambda}$.
 \end{enumerate}
	
The paper is organized as follows. In Section \ref{s1} we introduce the self-similarity transform (SST) and we write the problem \eqref{heatHardyHalf} in those variables. Section \ref{section:improvedHardy} consists of the proof of Theorem \ref{main_th1} i.e. the improved Hardy inequality in $\RR^{N,k}_+$. In Section \ref{functionalSpectral}, we provide the rigorous framework in which we pose the problem in order to prove the asymptotic estimates. Section \ref{section:asymptotic} contains the proof of the main result, namely Theorem \ref{main_th}.

	\section{The equation in similarity variables}\label{s1}
	In the spirit of \cite[Section 9]{VazquezZuazua}, we introduce the similarity variables $(s, y)\in (0, \infty)\times \RR_{+}^{N, k}$ through the following change of variables:
	$$(s, y):=\left(\ln(t+1), \frac{x}{\sqrt{t+1}}\right), \quad (t, x)=\left(e^{s}-1, e^\frac{s}{2}y\right),$$
	$$w(s,y):=e^\frac{Ns}{4}u\left(e^s-1,e^\frac{s}{2}y\right).$$
	An important feature of this transformation is that it keeps the $L^2$-norm invariant, i.e. 
	$$\|w(s)\|_{L^2}=\|u(t)\|_{L^2}.$$
	With these notations we have 
	$$\partial_s w(s,y)=e^\frac{Ns}{4}\left[\frac{N}{4}\cdot  u\left(e^s-1,e^\frac{s}{2}y\right) + e^\frac{s}{2}\frac{y}{2} \nabla u\left(e^s-1,e^\frac{s}{2}y\right) + e^s \partial_t u\left(e^s-1,e^\frac{s}{2}y\right)\right],$$
	$$\nabla_y w(s,y) = e^\frac{Ns}{4}\cdot e^\frac{s}{2} \nabla u\left(e^s-1,e^\frac{s}{2}y\right),$$
	$$\Delta_y w(s,y) = e^\frac{Ns}{4}\cdot e^{s} \Delta u\left(e^s-1,e^\frac{s}{2}y\right).$$
	
	Thus, problem \eqref{heatHamitonian} translates into:
	\begin{equation}
		\label{hardyPotentialSym}
		\left\{\begin{array}{ll}
			\partial_s w(s,y)+ A_\lambda w(s, y)-\frac{N}{4} w(s,y)  -\frac{y}{2} \nabla w(s,y)=0, & s> 0, \ y\in \RR_{+}^{N, k} \\
			w(s,y)=0, & s>0, y\in \partial \RR^{N,k}_+\\
			w(0,y)=u_0(y), & y\in \RR_{+}^{N, k}.
		\end{array}\right.
	\end{equation}
	Notice that the $s$-evolution generator of \eqref{hardyPotentialSym} gathering the non-zero order terms is an operator which can be written as:
	$$\Delta_y w+ \frac{y}{2}\cdot \nabla w=\frac{1}{K}\mathrm{div}(K \nabla w),$$
	where we recall that $K(y)=e^\frac{|y|^2}{4}$.
	This invites us to study the well-posedness of \eqref{hardyPotentialSym} in the weighted space $L^2(K)$. 
	For this reason,  we restrict  the initial data  $u_0$ to the weighted space $L^2(K)$. 
	However, since we prefer to work in a non-weighted space $L^2(\RR_+^{N,k})$, we 
	employ the transformation $v(s,y):=K^\frac{1}{2}(y)w(s,y)$ and study the equation satisfied by $v$, that is 
	\begin{equation}
		\label{hardyPotentialVhat}
		\left\{\begin{array}{ll}
			\partial_s v(s,y)+ A_\lambda v(s,y) + \frac{|y|^2}{16} v(s,y)=0, & s> 0,\  y\in \RR_{+}^{N, k}, \\
			v(s,y)=0, & s>0, y\in \partial \RR^{N,k}_+,\\
			v(0,y)=v_0(y)=K^\frac{1}{2}(y)w(0,y)=K^\frac{1}{2}(y)u_0(y), & y\in \RR_{+}^{N, k},
		\end{array}\right.
	\end{equation}
	with the initial data $v_0$ belongings to the space $L^2(\RR_+^{N, k})$.
	
	The Cauchy problem  \eqref{hardyPotentialVhat} is associated to the harmonic oscillator  operator with Hardy potential given by \eqref{LLambdaFormal}. 
	 Roughly speaking, we will prove that this operator has compact inverse. In particular, we are able to determine its first (positive) eigenvalue and the corresponding eigenfunctions.

    \section{Improved Hardy-Poincar\'{e} type innequality}
	\label{section:improvedHardy}
	This section is dedicated to the proof of Theorem \ref{main_th1}. The proof is very direct and short once we construct a positive eigenfunction in $\RR_+^{N, k}$ for the operator $L_\lambda$. This eigenfunction was defined in \eqref{minimizers}:  
	$$	\alpha_{k,\lambda}(x)=e^{-\frac{|x|^2}{8}}\,|x|^{m_\lambda-\frac{N-2}{2}}\, \frac{x_{N-k+1}x_{N-k+2}\cdots x_{N}}{|x|^k}\in L^2(\RR^{N,k}_+).$$ 
	First, notice the important detail that $\alpha_{k,\lambda}(x)>0$ in $\RR_{+}^{N, k}$. This allows to consider the well-defined function
	$$\phi(x):=\frac{u(x)}{\alpha_{k,\lambda}(x)},$$
	which belongs to $C_c^\infty(\RR_{+}^{N, k})$, provided that $u$ is smooth and compactly supported. Thus,
	$$\int_{\RR^{N,k}_+} |\nabla u|^2 \dx=\int_{\RR^{N,k}_+} \alpha_{k, \lambda}^2 |\nabla \phi|^2 \dx+ \int_{\RR^{N,k}_+} \phi^2 |\nabla \alpha_{k, \lambda}|^2 \dx+ 2 \int_{\RR^{N,k}_+} \alpha_{k, \lambda} \nabla \alpha_{k, \lambda} \phi\cdot \nabla \phi \dx.$$ 
	Due to the compact support of $\phi$, we can integrate by parts the mixt term above:
	\begin{align*}
		2 \int_{\RR^{N,k}_+} \alpha_{k, \lambda}
		\phi \nabla \alpha_{k, \lambda}\cdot \nabla \phi \dx &=  \int_{\RR^{N,k}_+} \alpha_{k, \lambda} \nabla \alpha_{k, \lambda} \cdot \nabla (\phi^2) \dx\\
		&=- \int_{\RR^{N,k}_+} \phi^2  \mathrm{div}(\alpha_{k, \lambda} \nabla\alpha_{k, \lambda}) \dx\\
		&=-\int_{\RR^{N,k}_+} \phi^2 |\nabla \alpha_{k, \lambda}|^2 \dx-\int_{\RR^{N,k}_+} \phi^2  \alpha_{k, \lambda}  \Delta \alpha_{k, \lambda}\dx.
	\end{align*}
	Therefore,
	\begin{align}\label{ident}
		\int_{\RR^{N,k}_+} |\nabla u|^2\dx &=\int_{\RR^{N,k}_+} \alpha_{k, \lambda}^2 |\nabla \phi|^2 \dx-\int_{\RR^{N,k}_+} \phi^2  \alpha_{k, \lambda}  \Delta \alpha_{k, \lambda}\dx\\ \nonumber
		&= \int_{\RR^{N,k}_+} \alpha_{k, \lambda}^2 \left|\nabla \left(\frac{u}{\alpha_{k, \lambda}}\right)\right|^2 \dx-\int_{\RR^{N,k}_+}  \frac{\Delta \alpha_{k, \lambda}}{\alpha_{k, \lambda}} |u|^2    \dx.
	\end{align}
	In addition, explicit computations show that  $\alpha_{k, \lambda}$ satisfies in classical sense the equation $L_\lambda \alpha_{k, \lambda}=\frac{1+m_\lambda}{2}\alpha_{k, \lambda}$, i.e.:
	\begin{equation}\label{miracle_eqn}
		\frac{-\Delta \alpha_{k, \lambda}(x)}{\alpha_{k, \lambda}(x)}=\frac{\lambda}{|x|^2}-\frac{|x|^2}{16}+\frac{1+m_\lambda}{2}, \quad \forall x\in \RR_{+}^{N,k}.
	\end{equation}
	We combine \eqref{ident} and \eqref{miracle_eqn} to obtain:
	\begin{equation}
	\label{lNormMinusL2norm}
	\begin{aligned}
		\int_{\RR_+^{N,k}}\left(|\nabla u|^2 -\lambda\frac{|u|^2}{|x|^2} + \frac{|x|^2}{16}|u|^2\right) \dx &- \frac{1+m_\lambda}{2}\int_{\RR^{N,k}_+} |u|^2 \dx \\
		&=\int_{\RR^{N,k}_+} \alpha_{k, \lambda}^2\left|\nabla\left(\frac{u}{\alpha_{k, \lambda}}\right)\right|^2 \dx \geq 0,
	\end{aligned}
	\end{equation}
		which proves \eqref{improved_Hardy}.
		
	In order to prove that the constant is sharp, we construct a sequence of functions $(\Lambda_\eps)_\eps\subset C_c^\infty(\RR_+^{N,k})$ such that:
 $$\lim_{\epsilon \to 0}\frac{l_\lambda [\Lambda_\eps]}{\|\Lambda_\eps\|^2_{L^2}}=\frac{1+m_\lambda}{2}.$$
It is enough to show that 
\begin{equation}
\label{conditionSharp.limitOfLlambda}
\lim_{\epsilon \to 0}\left(l_\lambda [\Lambda_\eps]-\frac{1+m_\lambda}{2}{\|\Lambda_\eps\|^2_{L^2}}\right)=0
\end{equation}
and 
\begin{equation}
\label{conditionSharp.L2normBoundedBelow}
\lim_{\epsilon \to 0}\|\Lambda_\eps\|_{L^2}>0.
\end{equation}
To construct such sequence we need the following Lemma:
\begin{lemma}\label{desing.aprox}
There exists a sequence of functions $(\psi_\eps)_{\eps>0}\in C_c^\infty(0, \infty)$  which satisfies:
	\begin{align}
		&\label{psi2}
		 {\rm supp}(\psi_\eps)\subseteq\left[\eps^2-\eps^4,\eps^{-2}+\eps^4\right];\\
		 &\label{psi3}
	0\leq 	\psi_\eps \leq 1,\, \psi_\eps\equiv 1\text{ on }\left[\eps+\eps^4,\eps^{-1}-\eps^4\right];\\
		 &\label{psi1}
		\text{for every }\zeta\geq 0\text{, }\lim_{\eps\to 0}\int_0^\infty e^{-\frac{r^2}{4}} r^{1+\zeta} |\psi'_\eps(r)|^2dr =0.
	\end{align}
	\end{lemma}
	
The sequence defined by  $\Lambda_\eps(x):=\alpha_{k,\lambda}(x)\,  \psi_\eps(|x|)$ satisfies \eqref{conditionSharp.limitOfLlambda}-\eqref{conditionSharp.L2normBoundedBelow}. Indeed,  $(\Lambda_\eps)_{\eps>0}$ converges to $\alpha_{k,\lambda}$ in $L^2(\RR^{N,k}_+)$ and, in view of \eqref{lNormMinusL2norm}, we get that:
 $$0\leq l_\lambda [\Lambda_\eps]-\frac{1+m_\lambda}{2}{\|\Lambda_\eps\|^2_{L^2}}=
 \int_{\RR_+^{N,k}}\alpha_{k,\lambda}^2|\nabla [\psi_\eps(|x|)]|^2 \leq |S^{N-1}| \int_0^\infty e^{-\frac{r^2}{4}}r^{1+2m_\lambda}|\psi_\eps'(r)|^2 dr,
 $$
 which tends to zero, by \eqref{conditionSharp.limitOfLlambda}. Furthermore, $\Lambda_\eps$ converges to $\alpha_{k,\lambda}$ in $L^2(\RR^{N,k}_+)$ as $\eps$ goes to zero.
 The proof of Theorem \ref{main_th1} is now finished.\hfill \qedsymbol{}
 
 \begin{proof}[Proof of Lemma \ref{desing.aprox}]In what follows we assume $\eps>0$ to be small enough. When $\zeta>0$, a simpler construction may be done. The main difficulty is to construct a sequence for $\zeta=0$, which corresponds to the critical value of $\lambda$. We provide an example which works for both cases.   
 The construction of $(\psi_\eps)_{\eps>0}$ goes as follows (the idea is inspired from \cite[Proposition 3.1]{CazacuFlynnLam}):
	First we consider the sequence   $\eta_\eps:\RR\to [0,\infty)$,
	$$\eta_\eps(r)=\left\{\begin{array}{cl}
		\frac{\log r-2\log\eps}{-\log\eps}, & r\in \left[\eps^2,\eps\right],\\[5pt]
		1, & r \in \left[\eps,\eps^{-1}\right],\\[5pt]
		-\frac{r\eps^2}{1-\eps}+\frac{1}{1-\eps}, &
		r \in \left[\eps^{-1},\eps^{-2}\right],\\[5pt]
		0, & \text{otherwise}.
	\end{array}\right.$$
	We notice that $\eta_\epsilon$ is continuous, $0\leq \eta_\eps \leq 1$, ${\rm supp}( \eta_\eps) =[\eps^2, \eps^{-2}]$  and its weak derivative is:
	\begin{equation}
	\label{EtaEpsPrim}
	\eta'_\eps(r)=\left\{\begin{array}{cl}
		-\frac{1}{r\log\eps}, & r\in (\eps^2,\eps),\\[5pt]
		-\frac{\eps^2}{1-\eps}, &
		r \in (\eps^{-1},\eps^{-2}),\\[5pt]
		0, & \text{otherwise}.
	\end{array}\right.
	\end{equation}
	In order to obtain a sequence $(\psi_\eps)_{\eps>0}$ of smooth functions, we consider $(\rho_{\eps})_{\eps>0}$ a standard sequence of mollifiers (\cite[p.108]{brezis}), $\rho_\eps(r):=\frac{1}{\eps^4}\rho\left(\frac{r}{\eps^4}\right)$, where 
 $\rho$ is a non-negative $C_c^\infty(\RR)$ function supported in $[-1,1]$ and having mass one.
  Then, 
${\rm supp}(\rho_\eps)\subseteq\left[-\eps^4,\eps^4\right]$.

	The desired sequence is obtained as the convolution $\psi_\eps=\eta_\eps*\rho_\eps$, which we prove that satisfies properties \eqref{psi2}-\eqref{psi1} above. Since $0\leq 
	\eta_\eps\leq 1$ and $\rho_\eps$ has mass one, we get
	$$\psi_\eps(r)=\int_{-\eps^4}^{\eps^4} \eta_\eps(r-s)\rho_\eps(s) ds\in [0,1]$$
	If $r\in [\eps+\eps^4,\eps^{-1}-\eps^4]$, then, for any $s\in [-\eps^4,\eps^4]$, $r-s\in [\eps,\eps^{-1}]$, so $\eta_\eps(r-s)=1$ and thus
	$\psi_\eps(r)= 1$.
	Similarly, one can prove that the support of $\psi_\eps$ is contained in $[\eps^2-\eps^4,\eps^{2}+\eps^4]$.
	
	We now to prove \eqref{psi1}.
	We notice that:
	\begin{equation}
	\label{PsiPrimConvolution}
	\psi_\eps'(r)=\int_{-\eps^4}^{\eps^4} \eta'_\eps(r-s)\rho_\eps(s) ds, 
	\end{equation}
	so we can do the same reasoning as above to get that:
	$${\rm supp}(\psi_\eps')\subseteq [\eps^2-\eps^4,\eps+\eps^4]\cup [\eps^{-1}-\eps^4,\eps^{-2}+\eps^4].$$ 
	When $r\in [\eps^2-\eps^4,\eps+\eps^4]$ and $s\in [-\eps^4,\eps^4]$, we have $r-s\geq r-\eps^4\geq r/2$, then the particular form of $\eta_\eps'$, given in \eqref{EtaEpsPrim}, together with \eqref{PsiPrimConvolution}, implies that,
	$$|\psi'(r)|\leq -\frac{1}{\log\eps}\int_{-\eps^4}^{\eps^4} \frac{1}{r-\eps^4}\rho_{\eps}(s)ds=-\frac{1}{(\log\eps)(r-\eps^4)}\leq \frac {-2}{r\log \eps }.$$
	For $r\in [\eps^{-1}-\eps^4, \eps^{-2}+\eps^4]$ a similar argument shows that 
	\[
	|\psi'(r)|\leq \frac{\eps^2}{1-\eps}.
	\]
	We split the integral in \eqref{psi1} into a sum $I_1(\eps)+I_2(\eps)$, where
	$$I_1(\eps)=\int_{\eps^2-\eps^4}^{\eps+\eps^4} e^{-\frac{r^2}{4}} r^{1+\zeta}  |\psi_\eps'(r)|^2dr$$
	and
	$$I_2(\eps)=\int_{\eps^{-1}-\eps^4}^{\eps^{-2}+\eps^4} e^{-\frac{r^2}{4}} r^{1+\zeta} |\psi_\eps'(r)|^2dr.$$
	It is clear that $I_2(\eps)$ tends to $0$ as $\eps$ approaches $0$, so we focus on $I_1(\eps)$. 	
	Therefore:
	$$I_1(\eps)\leq \frac{4}{(\log(\eps))^2}\int_{\eps^2-\eps^4}^{\eps+\eps^4} e^{-\frac{r^2}{4}} r^{\zeta-1}   dr.$$
	For $\zeta>0$, this immediately converges to $0$ as $\eps$ goes to $0$. For $\zeta=0$, we get that
	$$I_1(\eps)\leq 4 \frac{\log(2\eps)-\log(\eps^2/2)}{(\log(\eps))^2}=-\frac{4}{\log\eps},$$
	which also converges to $0$.
 \end{proof}
	
	\section{Functional and spectral framework}
	\label{functionalSpectral}
    One of the consequences of Theorem \ref{main_th1} is that the seminorm $\|\cdot\|_{l_\lambda}$ -- induced by $l_\lambda$ defined in \eqref{intro.defllambdaQuadratic} --  bounds from above the $L^2$ norm,
    \begin{equation}
    \label{lNormBoundsL2}
    \|u\|^2_{l_\lambda}\geq \frac{1+m_\lambda}{2} \|u\|^2_{L^2},
    \end{equation}
	for every $u\in C_c^{\infty}(\RR_+^{N,k})$.
	As a result, $\|\cdot\|_{l_\lambda}$ is actually a norm on $C_c^\infty(\RR_+^{N,k})$. We can extend it to a complete norm, by considering the space
	\begin{equation}
    \label{defH}
    H_\lambda=\overline{C_c^\infty(\RR_+^{N,k})}^{\|\cdot\|_{l_\lambda}},
    \end{equation}
	the closure with respect to the $\|\cdot\|_{l_\lambda}$ norm, which makes $H_\lambda$ a Hilbert space.
	More precisely, the estimate \eqref{lNormBoundsL2} implies that $H_\lambda$ can be defined as:
	$$H_\lambda=\left\{u\in L^2(\RR^{N,k}_+) : \exists (u_n)_n\subset C_c^\infty(\RR^{N,k}_+)\text{ Cauchy sequence in }\|\cdot\|_{l_\lambda}\text{ s.t. }\lim_{n\to \infty}\|u_n-u\|_{L^2}=0\right\}.$$

	 Inequality \eqref{lNormBoundsL2} implies that $H_\lambda$ is continuously embedded in $L^2$, which allows us to follow \cite[Theorem VIII.15]{ReedSimon} and define the operator $\mathcal{L}_\lambda: H_\lambda\to H_\lambda^*$, $\mathcal{L}_\lambda(u)= l_\lambda(u,\cdot)$, which is the Riesz isomorphism between $H_\lambda$ and $H_\lambda^*$, where $l_\lambda(\cdot,\cdot)$ is the bilinear form associated to $l_\lambda[\cdot]$ defined in \eqref{intro.defllambdaQuadratic}-\eqref{intro.defH}.
	 
The same result enables us to associate to the bilinear form $l_\lambda$ defined on $H_\lambda$ a self-adjoint unbounded operator $L_\lambda$ on $L^2$ with domain:
	\begin{equation}
	\label{def.DLlambda}
    D(L_\lambda)=\left\{u\in H_\lambda : \mathcal{L}_\lambda u \in L^2(\RR^{N,k}_+)\right\}.
	\end{equation}
	The integration by parts formula implies the operator $L_\lambda$ acts on $C_c^\infty(\RR^{N,k}_+)$ functions as follows:
	$$L_\lambda u=-\Delta u -\lambda \frac{u}{|x|^2}+\frac{|x|^2}{16}u,$$
	so we have successfully extended the operator in \eqref{LLambdaFormal} on the space $H_\lambda$.
	
	\begin{remark}
	    In fact, Theorem \ref{main_th1} asserts that 
	$$L_\lambda \geq \frac{1+m_\lambda}{2}, \quad \mathrm{where}\  m_\lambda=\sqrt{\lambda_{N, k}-\lambda},\  \lambda_{N, k}=\left(\frac{N-2}{2}+k\right)^2,$$
	in the sense of quadratic forms in $L^2(\RR_{+}^{N, k})$. 
	\end{remark}

	\subsection{A remark about the space $H_\lambda$ in the subcritical case}
	If $\lambda$ is subcritical, we can characterise more accurately the space $H_\lambda$ defined above:

	\begin{proposition}
		If $\lambda\in (-\infty,\lambda_{N,k})$, then
		$$H_\lambda={H_0^1}(\RR_+^{N,k})\cap L^2(\RR_+^{N,k};|\cdot|^2),$$
		that is:
		$$H_\lambda=\left\{u\in H_0^1(\RR_+^{N,k}) : \int_{\RR_+^{N,k}} |u(x)|^2\, |x|^2 dx<\infty\right\},$$
		and the norms $\|\cdot\|_{l_\lambda}$ and $\|\cdot\|_{H^1}+\|\cdot\|_{L^2(\RR_+^{N,k};|\cdot|^2)}$ are equivalent.\newline
	\end{proposition}
	\begin{proof}
		First, we prove that, on $C_c^\infty(\RR_+^{N,k})$, the two norms above are equivalent. Indeed, for $u\in C_c^\infty(\RR_+^{N,k})$, Hardy's inequality \eqref{Hardy_k} implies that:
		\[l_\lambda[u]\geq {\scriptstyle\left(1-\frac{\max\{0,\lambda\}}{\lambda_{N,k}}\right)}\int_{\RR^{N,k}_+} |\nabla u |^2\dx +\frac{1}{16}\|u\|^2_{L^2(\RR_+^{N,k};|\cdot|^2)}\]
		and
		 \[l_\lambda[u]\leq {\scriptstyle \left(1+\frac{\max\{0,-\lambda\}}{\lambda_{N,k}}\right)}\int_{\RR^{N,k}_+} |\nabla u |^2 \dx+ \frac{1}{16}\|u\|^2_{L^2(\RR_+^{N,k};|\cdot|^2)}.\]
		 
        These inequalities, together with Theorem \ref{main_th1}, imply the equivalence of the norms in the case of smooth compactly supported functions.
	
		The fact that $C_c^\infty(\RR_+^{N,k})$ is dense in 
		$H_0^1(\RR_+^{N,k})\cap L^2(\RR_+^{N,k}; |\cdot|^2)$
		 follows by a standard cut-off and mollification argument, so the Proposition is proved.
	\end{proof}
	\subsection{Compact embedding}
	In this section, we go further and prove that the embedding $H_\lambda\subset L^2$ is compact, which will lead to $L_\lambda$ having compact inverse.
	
	\begin{proposition}
		The embedding $H_\lambda\subset L^2(\RR_+^{N,k})$ is compact.
	\end{proposition}
	\begin{proof}
	    We proceed by contradiction. Let $(\psi_n)_{n\geq 1}$ be a bounded sequence in $H_\lambda$ which has no convergent subsequence in $L^2$. Since $H_\lambda$ is the closure of $C_c^\infty(\RR_+^{N,k})$ w.r.t. $\|\cdot\|_{l_\lambda}$ norm, we can assume that every $\psi_n$ is in $C_c^\infty(\RR_+^{N,k})$. Otherwise, we can choose the sequence $(\widehat{\psi}_n)_n\subset C_c^\infty(\RR_+^{N,k})$ such that $\|\psi_n-\widehat{\psi}_n\|_{l_\lambda}<\frac{1}{n}$, which is also bounded and converges in $L^2$ on a subsequence if and only if the original sequence $(\psi_n)_{n\geq 1}$ converges on a subsequence.
	    
	    It is standard to assume further that $\psi_n$ converges weakly to $0$ in $H_\lambda$, so also in $L^2$. Furthermore, since there is no subsequence of $(\psi_n)_{n\geq 1}$ that converges to $0$ in $L^2$,  then $\left(\|\psi_n\|_{L^2}\right)_n$ must be bounded from below away from $0$. Therefore, we can assume without losing generality that $\|\psi_n\|_{L^2}=1$. 
	    
	    Let us consider $M>0$ such that $\|\psi_n\|_{l_\lambda}\leq M,\forall n\geq 1$ i.e.:
		\begin{equation}
			\label{lNormBoundedM}
			\int_{\RR_+^{N,k}}\left[ |\nabla\psi_n|^2-\frac{\lambda}{|y|^2} |\psi_n|^2+\frac{|y|^2}{16}|\psi_n|^2\right] \dy\leq M,\forall n \geq 1.
		\end{equation}
		
	    In what follows, the general idea is to make use of the last term above in order to transfer the problem on a bounded domain, for which we know a compactness result.
	    First, using Hardy's inequality \eqref{Hardy_k}, the estimate \eqref{lNormBoundedM} implies that:
		$$\int_{\RR_+^{N,k}} \frac{|x|^2}{16}|\psi_n|^2\dx \leq M,\forall n\geq 1.$$
		Therefore, if we choose $R$ such that $\frac{16M}{R^2}<\frac{1}{2}$, we obtain:
		\begin{equation}
		\label{compactness.normPsi2}
		\int_{\RR_+^{N,k}\setminus B_R} |\psi_n|^2\dx \leq \frac{1}{2},\forall n\geq 1.
		\end{equation}
	    
	    Next, we construct a radial mollifier $\phi:\RR\to [0,\infty)$, which is smooth, $\phi\equiv 1$ on $(-\infty,0]$ and $\phi\equiv 0$ on $[1,\infty)$. Then, we consider, for every $R>0$,
		$\phi_R(x)=\phi(|x|-R).$
		The function $\phi_R$ satisfies:
		\begin{enumerate}[label=(\arabic*)]
			\item $\phi_R\in C_c^\infty(\RR^N)$, ${\rm supp}(\phi_R)\subseteq B_{R+1}$,
			\item $\phi_R\equiv 1$ on $B_R$,
			\item $|\phi_R|\leq 1$,
			\item $\nabla \phi_R$ and $\Delta \phi_R$ are bounded by a constant independent of $R$.
		\end{enumerate}
		Here, $B_R$ stands for the ball in $\RR^N$ with radius $R$, which is centered at the origin.
		With this properties in mind, we decompose every $\psi_n$:
		$$\psi_n=\psi_n\,  \phi_R +\psi_n\,  (1-\phi_R)=\psi_n^1+\psi_n^2.$$
		We obtain
		\begin{align}		\label{compactness.lLambdaNormBoundBelow}
		\|\psi_n\|_{l_\lambda}^2=& \int_{\RR_+^{N,k}}\left[|\nabla \psi_n^1|^2
  -\lambda\frac{|\psi_n^1|^2}{|x|^2}\right]\dx+\int_{\RR_+^{N,k}}\left[|\nabla \psi_n^2|^2\dx-\lambda\frac{|\psi_n^2|^2}{|x|^2}\right]\dx\\
&+		2\int_{\RR_+^{N,k}} \nabla \psi_n^1 \cdot \,  \nabla \psi_n^2 \dx-2\lambda\int_{\RR_+^{N,k}}\frac{|\psi_n|^2\, \phi_R(1-\phi_R)}{|x|^2}\dx +\frac{1}{16}\int_{\RR_+^{N,k}}  |\psi_n|^2\, |x|^2\dx.\nonumber
		\end{align}

		Further, since $\phi_R(1-\phi_R)$ is supported in $[R,R+1]$, we take $R$ is large enough such that $\frac{R^2}{16}\geq \frac{2\lambda}{|x|^2}$. As a result,
		$$-2\lambda\int_{\RR_+^{N,k}}\frac{|\psi_n|^2\phi_R(1-\phi_R)}{|x|^2} \dx+\frac{1}{16}\int_{\RR_+^{N,k}}  |\psi_n|^2\, |x|^2\dx\geq 0.$$
		Therefore, Hardy's inequality \eqref{Hardy_k} applied to $\psi_n^2$, together with \eqref{compactness.lLambdaNormBoundBelow}, implies that:
		$$\|\psi_n\|^2_{l_\lambda}\geq \int_{\RR_+^{N,k}}\left[|\nabla \psi_n^1|^2-\lambda \frac{|\psi_n^1|^2}{|x|^2}\right]\dx +2\int_{\RR_+^{N,k}} \nabla \psi_n^1 \cdot \nabla\psi_n^2\dx. $$
		From the definition of $\psi_n^1$ and $\psi_n^2$, we obtain, integrating by parts, that 
		$$2\int_{\RR_+^{N,k}} \nabla \psi_n^1 \cdot \nabla\psi_n^2\dx=\int_{\RR_+^{N,k}} |\nabla \psi_n|^2 \phi_R(1-\phi_R)\dx +\int_{\RR_+^{N,k}} |\psi_n|^2\left[\frac{1}{2}\Delta\phi_R-|\nabla\phi_R|^2\right]\dx.$$
		Therefore, 
		$$\|\psi_n\|_{l_\lambda}^2\geq \int_{\RR_+^{N,k}}\left[|\nabla \psi_n^1|^2-\lambda\frac{|\psi_n^1|^2}{|x|^2}\right]\dx +\int_{\RR_+^{N,k}} |\psi_n|^2\left[\frac{1}{2}\Delta\phi_R-|\nabla\phi_R|^2\right]\dx. $$
		Since $\left[\frac{1}{2}\Delta\phi_R-|\nabla\phi_R|^2\right]$ is bounded by a constant independent of $R$ and $\|\psi_n\|_{L^2}=1$, we deduce that for some positive constant $C$,
		$$\int_{\RR_+^{N,k}}\left[|\nabla \psi_n^1|^2-\lambda\frac{|\psi_n^1|^2}{|x|^2}\right]\dx\leq C, \ \forall n\geq 1.$$
	
		Since $\psi_n^1\in H_0^1(B_{R+1})$, it follows from 
		\cite[Theorem 2.2]{VazquezZuazua}  that $(\psi_n^1)_{n\geq 1}$ is bounded in any $W^{1,q}(B_{R+1})$ with $q\in [1,2)$. Choosing $q\in (\frac{2N}{N+2},2)$,  by Rellich-Kondrachov Theorem \cite[Theorem 9.16]{brezis}, $(\psi_n^1)_n$ converges strongly up to a subsequence in $L^2$ to some $\psi^1$.
	
		Next, since $\psi_n \rightharpoonup 0$ in $L^2$, and $\phi_R$ is bounded,
		$$\lim_{n\to \infty}\int_{\RR_{+}^{N, k}} \psi_n \, \phi_R\,\xi \dx =0, \forall \xi\in L^2.$$
		It follows that $\psi_n^1\rightharpoonup 0 $ in $L^2$ and the strong convergence result above implies that $\psi_n^1\to 0$ in $L^2$ up to a subsequence. This and 
		 \eqref{compactness.normPsi2} contradict $\|\psi_n\|_{L^2}=1$ for all $n\geq 1$. 
The conclusion follows.
	\end{proof}
    The compact embedding $H_\lambda\subset L^2$ can now be used to prove that the operator $L_\lambda$ has compact inverse.
	\begin{theorem}
		The operator $L_\lambda:D(L_\lambda)\to L^2(\RR^{N,k}_+)$ has compact inverse. More precisely, there exists a compact operator
		$L_\lambda^{-1}:L^2\to D(L_\lambda)$
		such that
		$$L_\lambda^{-1}\circ L_\lambda={\rm id}_{D(L_\lambda)}\hspace{0.3cm}\text{ and }\hspace{0.3cm}L_\lambda\circ L_\lambda^{-1}={\rm id}_{L^2}.$$
	\end{theorem}
	\begin{proof}
		We consider the following diagram:
		$$L^2\xhookrightarrow[cont.]{i_2} H_\lambda^*\xrightarrow{\mathcal{L}_\lambda^{-1}}H_\lambda\xhookrightarrow[comp.]{i_1}L^2,$$
		which gives us the desired operator $L_\lambda^{-1}$. 
		
		Set  $T:=i_1\circ \mathcal{L}_\lambda^{-1}\circ i_2$ be the operator defined using the diagram. In this setting, the definition \eqref{def.DLlambda} of $D(L_\lambda)$ can be written more exactly as:
		$$D(L_\lambda)=\{i_1(u): u\in H_\lambda\text{ s.t. } \exists v\in L^2\text{ satisfying }\mathcal{L}_\lambda u=i_2(v)\}.$$
		Therefore, it can be easily seen that
	$i_1(\mathcal{L}_\lambda^{-1}(i_2(L^2)))=D(L_\lambda),$ 
	    so the image of $T$ is exactly $D(L_\lambda)$.
		Further, for $u\in D(L_\lambda)$,
		$T(L_\lambda(u))=i_1(\mathcal{L}_\lambda^{-1}(i_2(\mathcal{L_\lambda}(u))))=u,$
		so $T\circ L_\lambda={\rm id}_{D(L_\lambda)}$. Similarly, we obtain $L_\lambda\circ T={\rm id}_{L^2}$.
		
		Finally, the operator $T$ is compact, as a composition of a compact operator with two continuous ones, so the proof is finished.
		
	\end{proof} 
	\subsection{The first eigenvalue of $L_\lambda$}
	\label{firstEigenvalue}
	In the previous section, we proved that the operator $L_\lambda$ is self-adjoint and has compact inverse. Therefore, its spectrum is discrete, the eigenvalues form an increasing sequence and they have corresponding eigenvectors which span a dense subspace of $L^2(\RR_+^{N,k})$. 
	
	Theorem \ref{main_th1} implies that no eigenvalue is less than $\frac{1+m_\lambda}{2}$, where we recall that $m_\lambda=\sqrt{\lambda_{N,k}-\lambda}$.
	In this section, we will prove that $\frac{1+m_\lambda}{2}$ is indeed an eigenvalue and we also find the corresponding eigenspace.
	\begin{theorem}
		\label{firstEigenfunction}
		The first eigenvalue of $L_\lambda$ is $\frac{1+m_\lambda}{2}$, whose eigenspace is spanned by the function $\alpha_{k,\lambda}$ defined in \eqref{minimizers}.
		\begin{proof}
			The proof consists of two steps:\\
			
			\textit{Step 1:} We prove that $\alpha_{k,\lambda} \in H_\lambda$ and $L_\lambda \alpha_{k,\lambda}=\frac{1+m_\lambda}{2} \alpha_{k,\lambda}$.\newline
			Indeed, with the notations in the proof of Theorem \ref{main_th1} (see Section \ref{section:improvedHardy}), $\Lambda_\eps(x)=\alpha_{k,\lambda}(x)\,\psi_\eps(|x|)$ belongs to $C_c^\infty(\RR_+^{N,k})$ and identity \eqref{lNormMinusL2norm} implies that:
			\begin{equation}
			   \label{LambdaEpsNorm} 
			\|\Lambda_\eps-\Lambda_\delta\|_{l_\lambda}=\frac{1+m_\lambda}{2}\|\Lambda_\eps-\Lambda_\delta\|_{L^2}+\int_{\RR_+^{N,k}} \alpha_{k,\lambda}^2(x)|\psi_\eps'(|x|)-\psi_\delta'(|x|)|^2 \dx
			\end{equation}
			As in the proof of Theorem \ref{main_th1}, 
			$\lim_{\eps \to 0}\|\Lambda_\eps-\alpha_{k,\lambda}\|_{L^2}= 0$
		and  $\int_{\RR^{N,k}_+} \alpha_{k,\lambda}^2(x)|\psi_\eps'(|x|)|^2\dx$ converges to zero. Thus,  \eqref{LambdaEpsNorm} implies that $(\Lambda_\eps)_\eps$ is a Cauchy sequence in the $\|\cdot\|_{l_\lambda}$ norm. It follows that the latter sequence is convergent in $H_\lambda$ and the limit must be $\alpha_{k,\lambda}$ (since $H_\lambda$ is continuously embedded in $L^2$). Therefore, $\alpha_{k,\lambda}\in H_\lambda$.
			
			In order to prove that $L_{\lambda}\alpha_{k, \lambda}=\frac{1+m_\lambda}{2}\alpha_{k, \lambda}$, we begin by fixing an arbitrary $\varphi\in C_c^\infty(\RR_+^{N,k})$. We recall that the operator $\mathcal{L}_\lambda:H_\lambda\to H_\lambda^*$ satisfies:
			$$\langle\mathcal{L}_\lambda u,v\rangle_{H_\lambda^*,H_\lambda}=l_\lambda(u,v),\forall u,v\in H_\lambda.$$
			As a result, using the convergence $\Lambda_\eps\to \alpha_{k,\lambda}$ in $H_\lambda$, we obtain that:
			\begin{equation}
				\label{LWeakConv}
				\lim_{\eps \to 0}\langle\mathcal{L}_\lambda\Lambda_\eps,\varphi\rangle_{H_\lambda^*,H_\lambda}= \langle\mathcal{L}_\lambda\alpha_{k,\lambda},\varphi\rangle_{H_\lambda^*,H_\lambda}.
			\end{equation}
			On the other hand, since $\Lambda_\eps(x)=\alpha_{k,\lambda}(x)\, \psi_\eps(|x|)\in C_c^\infty(\RR_+^{N,k})$, it follows that:
			\begin{align*}
			L_{\lambda}\Lambda_\eps&=-\Delta \Lambda_\eps - \frac{\lambda}{|x|^2}\Lambda_\eps+\frac{|x|^2}{16}\Lambda_\eps\\
			&=-\Delta (\alpha_{k,\lambda}(x)\,\psi_\eps(|x|)) - \frac{\lambda}{|x|^2}\alpha_{k,\lambda}(x)\,\psi_\eps(|x|)+\frac{|x|^2}{16}\alpha_{k,\lambda}(x)\, \psi_\eps(|x|).
			\end{align*}
			Let us choose $\eps$ small enough such that, by \eqref{psi3}, $\psi_\eps\equiv 1$ on the support of $\varphi$.  Thus
			$$(L_\lambda\Lambda_\eps,\varphi)_{L^2}=\left(-\Delta \alpha_{k,\lambda} - \frac{\lambda}{|x|^2}\alpha_{k,\lambda}+\frac{|x|^2}{16}\alpha_{k,\lambda},\varphi\right)_{L^2}.$$
			Now, \eqref{miracle_eqn} implies that, for $\eps$ small enough,
			$$\langle \mathcal{L}_\lambda\Lambda_\eps,\varphi\rangle_{H_\lambda^*,H_\lambda}=(L_\lambda \Lambda_\eps,\varphi)_{L^2}=\frac{1+m_\lambda}{2}(\alpha_{k, \lambda},\varphi)_{L^2}=\frac{1+m_\lambda}{2}\langle \alpha_{k, \lambda},\varphi \rangle_{H_\lambda^\star, H_\lambda},$$
			which implies, by \eqref{LWeakConv} and the density of $C_c^\infty(\RR_+^{N,k})$ in $H_\lambda$, that:
			$$\mathcal{L}_\lambda\alpha_{k,\lambda}=\frac{1+m_\lambda}{2}\alpha_{k,\lambda}.$$
			We have obtained that $\alpha_{k,\lambda}\in D(L_\lambda)$ and $L_\lambda\alpha_{k,\lambda}=\frac{1+m_\lambda}{2}\alpha_{k,\lambda}$.\\
			
			\textit{Step 2:} We prove that all eigenfunctions corresponding to the eigenvalue $\frac{1+m_\lambda}{2}$ are in the span of $\alpha_{k,\lambda}$.\\
			Indeed, let $u\in D(L_\lambda)$ such that $L_\lambda u=\frac{1+m_\lambda}{2}u$. Since $u\in H_\lambda$, there exists a sequence $(u_n)_{n\geq 1}$ approximating $u$ in the $\|\cdot\|_{l_\lambda}$ norm and hence in $L^2$.
			By \eqref{lNormMinusL2norm},
			$$\|u_n\|_{l_\lambda}^2-\frac{1+m_\lambda}{2}\|u_n\|_{L^2}^2=\int_{\RR_+^{N,k}} 
			\alpha_{k,\lambda}^2 \left|\nabla\left(\frac{u_n}{\alpha_{k,\lambda}}\right)\right|^2\dx.$$
			Letting $n\rightarrow \infty$, the left hand side converges to zero since 
			\[
			\|u\|_{l_\lambda}^2-\frac{1+m_\lambda}2\|u\|^2_{L^2}=(L_\lambda u,u)_{L^2}-\frac{1+m_\lambda}2\|u\|^2_{L^2}=0.
			\]
			Hence
			\begin{equation}
			\label{limANKGrad}
			\lim_{n\to \infty}\int_{\RR_+^{N,k}} \alpha_{k,\lambda}^2 \left|\nabla\left(\frac{u_n}{\alpha_{k,\lambda}}\right)\right|^2 \dx =0.
			\end{equation}
			With the notations
			$$\xi_n=\frac{u_n}{\alpha_{k,\lambda}} \text{ and }\xi=\frac{u}{\alpha_{k,\lambda}},$$
			we obtain 
			$$\lim_{n\to \infty} \int_{\RR_+^{N,k}} \alpha_{k,\lambda}^2 |\nabla \xi_n|^2 \dx=0$$
			and
			$$\lim_{n\to \infty} \int_{\RR_+^{N,k}} \alpha_{k,\lambda}^2 |\xi_n-\xi|^2 \dx=\lim_{n\to \infty} \|u_n-u\|^2_{L^2}=0.$$
	
			Next, since $\alpha_{k,\lambda}$ is a positive continuous function, we obtain that, 
	$\xi_n\rightarrow \xi$ and $\nabla \xi_n\rightharpoonup 0$ in $L^2_{loc}(\RR^{N,k}_+)$.	
			The above limits imply that the $L^2$ function $\xi$ has null weak gradient on every compact set, so $\xi$ is a constant function on the whole $\RR_+^{N,k}$. 
Therefore, $u\in {\rm Span}(\alpha_{k,\lambda})$, so the proof is finished.
		\end{proof}
	\end{theorem}
	
	\section{Asymptotic behaviour for the heat equation with Hardy potential on $\RR_+^{N,k}$}
	\label{section:asymptotic}
	In this section, we will analyse the long-time decay and the asymptotic profile for the solutions of the Cauchy problem \eqref{heatHardyHalf}, by first computing the profile for the problem in self-similarity variables \eqref{hardyPotentialSymFull}.
	
	With the definitions in Section \ref{functionalSpectral},  equation \eqref{hardyPotentialSymFull} can be written as:
	\begin{equation}
		\label{equationVHat}
		\begin{cases}
			\partial_s v(s,y)+L_\lambda v(s,y)=0, &s> 0, y\in \RR^{N,k}_+,\\
			v(0,y)=v_0(y)=K^\frac{1}{2}(y)u_0(y), &y\in \RR^{N,k}_+.
		\end{cases}
	\end{equation}
	By the spectral results in Section \ref{firstEigenvalue}, there exists an orthonormal sequence of eigenfunctions $(e_n)_{n\geq 0}$ of $L_\lambda$, which span a dense subset of $L^2$ and the corresponding sequence of eigenvalues $(\mu_n)_{n\geq 0}$ is non-decreasing, $\mu_0=\frac{1+m_\lambda}{2}$ and $\mu_1>\frac{1+m_\lambda}{2}$, where we recall that $m_\lambda=\sqrt{\lambda_{N,k}-\lambda}$.\\
	Moreover, by Theorem \ref{firstEigenfunction}, the eigenspace corresponding to $\mu_0$ is spanned by:
	$$\alpha_{k,\lambda}(x)=e^{-\frac{|x|^2}{8}}\, |x|^{m_\lambda-\frac{N-2}{2}}\, \frac{x_{N-k+1}x_{N-k+2}\cdots x_{N}}{|x|^k}.$$
	Then, the first normalised eigenfunction is given by:
	$$e_0=\left(\left\|\alpha_{k,\lambda}\right\|_{L^2}\right)^{-1}\alpha_{k,\lambda}.$$
	Since the function $\alpha_{k,\lambda}^2$ is  evenly symmetric in each variable, we have that:
	$$\|\alpha_{k,\lambda}\|_{L^2(\RR_+^{N,k})}^2=\frac{1}{2^k}\int_{\RR^N} e^{-\frac{|x|^2}{4}}\, |x|^{2m_\lambda-N+2} \left(\frac{x_{N-k+1}x_{N-k+2}\cdots x_{N}}{|x|^k}\right)^2 \dx.$$
	Therefore, by passing to polar coordinates, we obtain:
	\begin{equation}
	\label{L2normOfAlphaNK}
	\|\alpha_{k,\lambda}\|_{L^2(\RR_+^{N,k})}^2=\frac{\omega_{N,k}}{2^k}\int_0^\infty e^{-\frac{r^2}{4}}\, r^{2m_\lambda+1} dr,
	\end{equation}
	where $\omega_{N,k}=\int_{S^{N-1}}\sigma_{N-k+1}^2\cdots \sigma_N^2 d\sigma$.
	In particular, for $\lambda=\lambda_{N,k}$ critical,
	$$\|\alpha_{k,\lambda}\|_{L^2(\RR_+^{N,k})}^2=\frac{\omega_{N,k}}{2^{k-1}},$$
	so, in this case,
	$$e_0=\sqrt{\frac{2^{k-1}}{\omega_{N,k}}}\alpha_{k,\lambda}.$$

\subsection{Proof of Theorem \ref{main_th}}
	We begin by computing the solution of \eqref{equationVHat} explicitly using the Hilbert basis $(e_n)_{n\geq 0}$.	First, we decompose the initial data $v_0\in L^2(\RR_+^{N,k})$ with respect to this basis:
	\begin{equation}\label{expansionInitialData}
	v_0(y)=\sum_{n\geq 0} \beta_n(0) \, e_n(y),\quad \text{where}\quad
	\beta_n(0)=\int_{\RR_+^{N,k}} v_0\, e_n\, \dy.
	\end{equation}
	It follows that the solution $v\in C([0,\infty),L^2(\RR_+^{N,k}))$ of \eqref{equationVHat} is given by
	\begin{equation}
	\label{solutionHeatHardySymFullWithSpectral}
	v(s,y)=\sum_{n\geq 0} \beta_n(0) e^{-\mu_n s}\, e_n(y).
	\end{equation}
 By Parseval's identity we have:
 \begin{equation}\label{Parseval}
 \|v(s)\|_{L^2}^2 = \sum_{n\geq 0} \beta_n^2 (0) e^{-2\mu_n s}\leq e^{-2 \mu_0 s}\sum_{n\geq 0}\beta_n^2(0)=e^{-2 \mu_0 s} \|v_0\|_{L^2}^2.
 \end{equation}
	Since $\mu_0=\frac{m_\lambda+1}{2}$, it follows that:
	\begin{equation}
	\label{decay.estimate1}
	e^{\frac{1+m_\lambda}{2}s}\|v(s)\|_{L^2}\leq \|v_0\|_{L^2}.
	\end{equation}
    Similarly, by the fact that $\mu_n\geq \mu_1>\frac{m_\lambda+1}{2}$ for every $n\geq 1$, we obtain:
   $$ \|v(s)-\beta_0(0) e^{-\frac{1+m_\lambda}{2}s} e_0\|^2_{L^2}=\sum_{n\geq 1}\beta_n^2(0) e^{-2\mu_n s}\leq e^{-(2\mu_1-(1+m_\lambda))s}e^{-(1+m_\lambda)s}\|v_0\|^2_{L^2}$$
   As a result,
	\begin{equation}\label{decay.profile1}
	\lim_{s\to\infty} e^{\frac{1+m_\lambda}{2}s}\left\|v(s)-\beta_0(0)\, e^{-\frac{1+m_\lambda}{2}s}\,  e_0\right\|_{L^2}=0.
	\end{equation}
	Taking into account that:
 $$v(s,y)=e^{\frac{|y|^2}{8}}e^{\frac{Ns}{4}}u(e^s-1,e^\frac{s}{2}y) \hspace{0.1cm}\text{ and }\hspace{0.1cm} \alpha_{k, \lambda}(y)= (t+1)^{\frac{N}{4}-\frac{1+m_\lambda}{2}}e^{-\frac{|x|^2}{8(t+1) }} e^{\frac{|x|^2}{8}}\alpha_{k, \lambda}(x),$$
  estimate \eqref{decay.estimate1} implies that:
	$$(t+1)^{1+m_\lambda} \int_{\RR_+^{N, k}} \left|u(t,x)\right|^2 e^{\frac{|x|^2}{4(t+1)}}\dx\leq \int_{\RR_+^{N,k}}\left|u_0(x)\right|^2\, e^{\frac{|x|^2}{4}}\dx$$
	and \eqref{decay.profile1} reads as:
		$$\lim_{t\to \infty}(t+1)^{1+m_\lambda} \int_{\RR_+^{N,k}} \left|u(t,x)-\beta\, (t+1)^{-(1+m_\lambda)}e^{-\frac{|x|^2}{4(t+1)}}\, e^{\frac{|x|^2}{8}}\alpha_{k, \lambda}(x)\right|^2 e^{\frac{|x|^2}{4(t+1)}}\dx=0,$$
		where we denote $\beta:=\left(\left\|\alpha_{k,\lambda}\right\|_{L^2(\RR_{+}^{N, k})}\right)^{-1}\beta_0(0)$.\\
    Since, for any $t$ and $x$, we have that $e^{\frac{|x|^2}{4(t+1)}}\geq 1$, \eqref{decay.estimate1} and \eqref{decay.profile2} lead to:
    	\begin{equation}
    	\label{decay.estimate2}
    	(t+1)^{1+m_\lambda} \int_{\RR_+^{N,k}} \left|u(t,x)\right|^2 \dx\leq \int_{\RR^{N,k}_+}\left|u_0(x)\right|^2 e^{\frac{|x|^2}{4}} \dx
    	\end{equation}
    and 
    \begin{equation}
    \label{decay.profile2}
    \lim_{t\to \infty} t^{1+m_\lambda} \int_{\RR_+^{N,k}} \left|u(t,x)-\beta\, (t+1)^{-(1+m_\lambda)}e^{-\frac{|x|^2}{4(t+1)}}\, e^{\frac{|x|^2}{8}}\alpha_{k, \lambda}(x)\right|^2 \dx= 0.
    \end{equation}
    Taking into account that, due to a change of variables $x=z\sqrt{t}$ followed by dominated convergence,
    $$\lim_{t\to \infty} t^{1+m_\lambda} \int_{\RR_+^{N,k}} \left|\left[(t+1)^{-(1+m_\lambda)}e^{-\frac{|x|^2}{4(t+1)}}-t^{-(1+m_\lambda)}e^{-\frac{|x|^2}{4t}}\right] e^{\frac{|x|^2}{8}}\alpha_{k, \lambda}(x)\right|^2 \dx= 0,$$
    we arrive to the asymptotic results in Theorem \ref{main_th}.
    
    We note that \eqref{expansionInitialData} leads to:
 $$\beta=\left(\left\|\alpha_{k,\lambda}\right\|_{L^2(\RR_{+}^{N, k})}\right)^{-1}\int_{\RR_{+}^{N, k}} v_0(x) \alpha_{k, \lambda}(x) \dx.=\left(\left\|\alpha_{k,\lambda}\right\|_{L^2(\RR_{+}^{N, k})}\right)^{-1}\int_{\RR_{+}^{N, k}} u_0(x)e^{\frac{|x|^2}{8}} \alpha_{k, \lambda}(x) \dx.$$
 
 In order to prove the optimality of the polynomial decay obtained in \eqref{decay.estimate2}, we consider the initial data: 
 $$u_0(x)=e^{-\frac{|x|^2}{8}}\alpha_{k,\lambda}(x),$$
 so that $v_0$ is exactly the eigenfunction $\alpha_{k,\lambda}$ of $L_\lambda$. Therefore, the solution $v$ of \eqref{equationVHat} is given by:
 $$v(s,y)=e^{-\frac{1+m_\lambda}{2}s} v_0(y).$$
 Consequently, reverting the self-similarity change of variables \eqref{selfSimilarityFull}, the corresponding solution of \eqref{heatHardyHalf} is:
 $$
 \begin{aligned}
 u(t,x)&=e^{-\frac{|x|^2}{8(t+1)}}(t+1)^{-\frac{N}{4}}\, v\left(\log(t+1),\frac{x}{\sqrt{t+1}}\right)\\
 &=(t+1)^{-\frac{1+m_\lambda}{2}-\frac{N}{4}}e^{-\frac{|x|^2}{8(t+1)}}v_0\left(\frac{x}{\sqrt{t+1}}\right)\\
 &=(t+1)^{-\frac{1+m_\lambda}{2}-\frac{N}{4}}u_0\left(\frac{x}{\sqrt{t+1}}\right).
 \end{aligned}
 $$
Finally, after another change of variable $x=z\sqrt{t+1}$, we arrive to:
 $$\|u(t)\|_{L^2(\RR^{N,k}_+)}= (t+1)^{-\frac{1+m_\lambda}{2}} \|u_0\|_{L^2(\RR^{N,k}_+)},$$
 so the optimal decay of the solutions of \eqref{heatHardyHalf} is $\gamma_\lambda=\frac{1+m_\lambda}{2}$.
 \hfill\qedsymbol{}
	\begin{remark}
	The solution \eqref{solutionHeatHardySymFullWithSpectral} of \eqref{hardyPotentialSymFull} obtained via the spectral method above is the same with the one obtained in  \eqref{well-posedness} via semigroup theory. 
	\end{remark}
	\begin{proof}
	From the expression in \eqref{solutionHeatHardySymFullWithSpectral}, we deduce that $v\in C([0,\infty),L^2(\RR^{N,k}_+))$. Moreover, the fact that the sequence $(e_k)_k$ is orthonormal in  $L^2$ implies that, for $s_0>0$,
	$$\lim_{s\to s_0} \frac{\|v(s)-v(s_0)+(s-s_0)\sum_{n\geq 0}\beta_n(0)\mu_n\, e^{-\mu_n s_0}e_n\|_{L^2(\RR^{N,k}_+)}}{s-s_0}=0.$$
	Therefore, 
	$$\partial_s v(s)=-\sum_{n\geq 0}\beta_n(0)\mu_n e^{-\mu_n s}e_k,$$ 
	from which we deduce that $v\in C^1((0,\infty),L^2(\RR^{N,k}_+))$.
	
	We are left to prove that $v\in C((0,\infty),D(L_\lambda))$ and, that, for $s>0$, $L_\lambda v(s)=\sum_{n\geq 0} \beta_n(0)\mu_n\, e^{-\mu_n s}e_n$. Indeed, since
	$$l_\lambda(e_i,e_j)=(L_\lambda e_i,e_j)_{L^2}=\mu_i\delta_{i,j},\forall i,j\geq 0,$$
	we deduce that
	$$\left\|\sum_{i=n_1}^{n_2} \beta_i(0)e^{-\mu_i s} e_i\right\|^2_{l_\lambda}=\sum_{i=n_1}^{n_2} \beta_i(0)^2\mu_i e^{-2\mu_i s}.$$
	Therefore, as $\sum_{n\geq 0} \beta_n(0)^2<\infty$, the sequence $(s_n)_n=\left(\sum_{i=0}^{n} \beta_i(0)e^{-\mu_i s} e_i\right)_n$ is Cauchy in $\|\cdot\|_{l_\lambda}$. 
 
	Since the space $H_\lambda$ is continuously embedded in $L^2(\RR^{N,k}_+)$, it follows that the limit of $(s_n)_n$ in $H_\lambda$ coincides with the $L^2$ limit of $(s_n)_n$, which is $v(s)$. Therefore, $v(s)\in H_\lambda$, for every $s>0$.
	
	Moreover, the following limit takes place in $H_\lambda^*$:
	\begin{equation}
	\label{equivalenceWellPosedness.limit1}
	    \mathcal{L}_\lambda v(s)=\lim_{n\to \infty} \mathcal{L}_\lambda s_n
	\end{equation}
    On the other hand, since every $e_i$ is in $D(L_\lambda)$, then
    $$\mathcal{L}_\lambda s_{n}=L_\lambda s_n=\sum_{i=0}^n \beta_i(0)\mu_i e^{-\mu_i s} e_i.$$
    Form the expression above, we deduce that $(L_\lambda s_n)_n$ is a Cauchy sequence in $L^2$. Therefore, \eqref{equivalenceWellPosedness.limit1}, together with the continuous embedding $L^2\subset H_\lambda^*$ implies that:
    $$\mathcal{L}_\lambda v(s)=\sum_{n\geq 0} \beta_n(0)\mu_n e^{-\mu_n s} e_n\in L^2,$$
	so we conclude by \eqref{def.DLlambda} that,for every $s>0$, $v(s)\in D(L_\lambda)$ and 
	 $$L_\lambda v(s)=\sum_{n\geq 0} \beta_n(0)\mu_n e^{-\mu_n s} e_n.$$
	\end{proof}

    \providecommand{\bysame}{\leavevmode\hbox to3em{\hrulefill}\thinspace}
\providecommand{\MR}{\relax\ifhmode\unskip\space\fi MR }

\providecommand{\MRhref}[2]{%
  \href{http://www.ams.org/mathscinet-getitem?mr=#1}{#2}
}
\providecommand{\href}[2]{#2}


\begin{thebibliography}{10}

\bibitem{AbdellaouiPeralPrimo}
Boumediene Abdellaoui, Ireneo Peral, and Ana Primo, \emph{Strong regularizing
  effect of a gradient term in the heat equation with the {H}ardy potential},
  J. Funct. Anal. \textbf{258} (2010), no.~4, 1247--1272. \MR{2565839}

\bibitem{BarasGoldstein}
Pierre Baras and Jerome~A. Goldstein, \emph{The heat equation with a singular
  potential}, Trans. Amer. Math. Soc. \textbf{284} (1984), no.~1, 121--139.
  \MR{742415}

\bibitem{brezis}
Haim Brezis, \emph{Functional analysis, Sobolev spaces and partial differential
  equations}, Universitext, Springer New York, 2010.
  
\bibitem{BM1997}
Ha\"{\i}m Brezis and Moshe Marcus, \emph{Hardy's inequalities revisited},
  vol.~25, 1997, Dedicated to Ennio De Giorgi, pp.~217--237.
  
\bibitem{BV1997}
Ha\"{\i}m Brezis and Juan~Luis V\'{a}zquez, \emph{Blow-up solutions of some
  nonlinear elliptic problems}, Rev. Mat. Univ. Complut. Madrid \textbf{10}
  (1997), no.~2, 443--469.

\bibitem{CazacuFlynnLam}
Cristian Cazacu, Joshua Flynn, and Nguyen Lam, \emph{Short proofs of refined
  sharp Caffarelli-Kohn-Nirenberg inequalities}, Journal of Differential
  Equations \textbf{302} (2021), 533--549.

\bibitem{CazacuDavid}
Cristian Cazacu and David Krejčiřík, \emph{The Hardy inequality and the heat
  equation with magnetic field in any dimension}, Communications in Partial
  Differential Equations \textbf{41} (2016), no.~7, 1056--1088.

\bibitem{cazenave}
Thierry Cazenave and Alain Haraux, \emph{An introduction to semilinear
  evolution equations}, Oxford University Press, 1998.

\bibitem{EscobedoKavian}
M.~Escobedo and O.~Kavian, \emph{Variational problems related to self-similar
  solutions of the heat equation}, Nonlinear Analysis: Theory, Methods \&
  Applications \textbf{11} (1987), no.~10, 1103--1133.

\bibitem{EscobedoZuazua1991}
Miguel Escobedo and Enrike Zuazua, \emph{Large time behavior for
  convection-diffusion equations in $\RR^N$}, Journal of Functional Analysis
  \textbf{100} (1991), no.~1, 119--161.

\bibitem{FerreiraMesquita}
Lucas C.~F. Ferreira and Cl\'{a}udia Aline A.~S. Mesquita, \emph{An approach
  without using {H}ardy inequality for the linear heat equation with singular
  potential}, Commun. Contemp. Math. \textbf{17} (2015), no.~5, 1550041, 16.
  \MR{3404752}

\bibitem{GarciaAzoreroPeralAlonso}
J.~P. Garc\'{\i}a~Azorero and I.~Peral~Alonso, \emph{Hardy inequalities and
  some critical elliptic and parabolic problems}, J. Differential Equations
  \textbf{144} (1998), no.~2, 441--476. \MR{1616905}
  
\bibitem{GM2011}
Nassif Ghoussoub and Amir Moradifam, \emph{Bessel pairs and optimal {H}ardy and
  {H}ardy-{R}ellich inequalities}, Math. Ann. \textbf{349} (2011), no.~1,
  1--57.

\bibitem{GidasSpruck}
B.~Gidas and J.~Spruck, \emph{Global and local behavior of positive solutions
  of nonlinear elliptic equations}, Communications on Pure and Applied
  Mathematics \textbf{34} (1981), no.~4, 525--598.

\bibitem{GigaKohn}
Yoshikazu Giga and Robert~V. Kohn, \emph{Asymptotically self-similar blow-up of
  semilinear heat equations}, Communications on Pure and Applied Mathematics
  \textbf{38} (1985), no.~3, 297--319.

\bibitem{Gkikas}
Konstantinos~T. Gkikas, \emph{Hardy-{S}obolev inequalities in unbounded domains
  and heat kernel estimates}, J. Funct. Anal. \textbf{264} (2013), no.~3,
  837--893. \MR{3003739}


\bibitem{IshigeMukai}
Kazuhiro Ishige and Asato Mukai, \emph{Large time behavior of solutions of the
  heat equation with inverse square potential}, Discrete Contin. Dyn. Syst.
  \textbf{38} (2018), no.~8, 4041--4069. \MR{3814364}

\bibitem{Krejcirik}
David Krej\v{c}i\v{r}\'{\i}k, \emph{The improved decay rate for the heat
  semigroup with local magnetic field in the plane}, Calc. Var. Partial
  Differential Equations \textbf{47} (2013), no.~1-2, 207--226. \MR{3044137}
  
\bibitem{DavidCurvedWedges}
David Krejčiřík, \emph{The {H}ardy inequality and the heat flow in curved
  wedges}, Port. Math. \textbf{73} (2016), no.~2, 91--113. \MR{3500825}

\bibitem{DavidZuazuaTwistedDomains}
David Krejčiřík and Enrique Zuazua, \emph{The Hardy inequality and the heat
  equation in twisted tubes}, Journal de Mathématiques Pures et Appliquées
  \textbf{94} (2010), no.~3, 277--303.

\bibitem{LLZ2020}
Nguyen Lam, Guozhen Lu, and Lu~Zhang, \emph{Geometric {H}ardy's inequalities
  with general distance functions}, J. Funct. Anal. \textbf{279} (2020), no.~8,
  108673, 35.

\bibitem{QianShen}
Chenyin Qian and Zifei Shen, \emph{Existence of global solutions and attractors
  for the parabolic equation with critical {S}obolev and {H}ardy exponent in
  {$\Bbb R^N$}}, Nonlinear Anal. Real World Appl. \textbf{42} (2018), 290--307.
  \MR{3773361}

\bibitem{ReedSimon}
Michael Reed and Barry Simon, \emph{Methods of modern mathematical physics},
  Academic Press, 1972.

\bibitem{Su-Yang}
Dan Su and Qiao-Hua Yang, \emph{On the best constants of {H}ardy inequality in
  {$\Bbb R^{n-k}\times(\Bbb R_+)^k$} and related improvements}, J. Math. Anal.
  Appl. \textbf{389} (2012), no.~1, 48--53. \MR{2876479}

\bibitem{VazquezZuazua}
Juan~Luis Vazquez and Enrike Zuazua, \emph{The Hardy inequality and the
  asymptotic behaviour of the heat equation with an inverse-square potential},
  Journal of Functional Analysis \textbf{173} (2000), no.~1, 103--153.

\end{thebibliography}
\end{document}